%% file: log.tex
\documentclass[12pt,twoside]{amsart}
\usepackage{amsmath,amsfonts,amssymb,amsthm,enumitem,comment,dirtytalk,tikz-cd,booktabs,makecell}

\renewcommand{\leq}{\leqslant}
\renewcommand{\geq}{\geqslant}
 
\newtheorem{theorem}{Theorem}[section]

\newtheorem{proposition}[theorem]{Proposition}

\newtheorem{corollary}[theorem]{Corollary}

\newtheorem{lemma}[theorem]{Lemma}

\newtheorem{remark}[theorem]{Remark}

\theoremstyle{definition}
\newtheorem{definition}[theorem]{Definition}

\title{Logarithmic bounds on Fujita's conjecture}
\author{Luca Ghidelli} \author{Justin Lacini}
\date{}

\address{Mathematisches Institut, Georg-August-Universit\"at G\"ottingen, Bunsenstr. 3-5, D-37073 G\"ottingen, Germany}
\email{lghidel@mathematik.uni-goettingen.de}
\address{Department of Mathematics, University of Kansas, 1450 Jayhawk Blvd. , Lawrence, KS 66045, USA}
\email{jlacini@ku.edu}

\begin{document}

\begin{abstract}
Let $X$ be a smooth complex projective variety of dimension $n$. We prove bounds on Fujita's basepoint freeness conjecture that grow as $n\operatorname{log}\operatorname{log}(n)$. 
\end{abstract}

\maketitle

\section{Introduction}

The purpose of this paper is to prove the following result:

\begin{theorem}\label{intro1}
Let $X$ be a smooth projective variety of dimension $n$ defined over an algebraically closed field of characteristic zero and let $L$ be an ample line bundle on $X$. 
Then $K_X + mL$ is basepoint free for any positive integer $m\geqslant \operatorname{max}\{n+1, n(\operatorname{log}\operatorname{log}(n)+2.34)\}$. 
\end{theorem}

A conjecture of Fujita \cite{fujita1} states that, in the hypothesis of Theorem \ref{intro1}, $K_X+mL$ is basepoint free for all $m\geqslant n+1$. Since maps to projective space are one
of the main tools used in the study of projective varieties, Fujita's conjecture has received considerable attention. 
Reider proved the conjecture for surfaces in \cite{reider} shorty after its formulation by using Bogomolov's instability theorem for rank two vector bundles.
Ein and Lazarsfeld proved it for threefolds in \cite{einlazarsfeld1} by introducing techniques from the Minimal Model Program.
Later, Kawamata proved the conjecture for fourfolds in \cite{kawamata1} and Ye and Zhu recently proved it for fivefolds in \cite{yezhu1} and \cite{yezhu2}.

For sixfolds, we prove the following:

\begin{theorem}\label{intro2}
Let $X$ be a smooth projective variety of dimension six defined over an algebraically closed field of characteristic zero and let $L$ be an ample line bundle on $X$. 
Then $K_X+mL$ is basepoint free for any positive integer $m\geqslant 8$. 
\end{theorem}

While these sporadic cases may be considered as evidence for the conjecture to hold true in general, in higher dimensions much less is known. 
The first general result is due to Angehrn and Siu, who used techniques of analytic algebraic geometry to prove that $K_X+mL$ is basepoint free for all $m\geqslant (n^2+n+2)/2$ in \cite{angehrnsiu}. 
Koll\'{a}r adapted their proof to the algebraic setting in \cite{pairs}. By using a different idea, later Helmke \cite{helmke1} 
also established a general method that essentially leads to a quadratic bound. Heier \cite{heier} 
combined Angehrn-Siu's approach and Helmke's approach to give a bound that is $O(n^{4/3})$. 

Once one knows that the linear series $|K_X+mL|$ gives a morphism to projective space, some questions naturally arise. For instance, it is interesting
to know if the morphism is birational. More in general, we say that $K_X+mL$ separates $r$ points if the restriction morphism
\[
H^0(X, K_X+mL)\rightarrow H^0(T, (K_X+mL)|_T)
\]
is surjective for any reduced subscheme $T$ of length $r$. Naturally, $K_X+mL$ separates two points if and only if
the morphism defined by $|K_X+mL|$ is birational. Angehrn and Siu \cite{angehrnsiu} (see also \cite{pairs}) showed that
$K_X+mL$ separates $r$ points for $m\geqslant (n^2 +2rn-n+2)/2$.
In this direction we prove:

\begin{theorem}\label{intro3}
Let $X$ be a smooth projective variety of dimension $n$ defined over an algebraically closed field of characteristic zero and let $L$ be an ample line bundle.
Then $K_X+mL$ separates $r$ points for any positive integer $m\geqslant r+n-1+\sqrt r \, n (\log \log n + 2.34)$.
\end{theorem}

Let us now briefly explain the ideas behind the proofs of Theorems \ref{intro1}-\ref{intro3}.
Let $x$ be a point in a smooth projective variety $X$ and let $L$ be an ample line bundle. Suppose that we wish to find a section of $H^0(X,K_X+L)$ that does not vanish at $x$.
If $L^n$ is large enough,  a well established method in birational geometry is to find a divisor $D\in |L|_{\mathbb{Q}}$ that has large order of vanishing at 
$x$. In this way the pair $(X,D)$ has a non Kawamata log terminal center $Z$ containing $x$, and then one tries to use vanishing theorems to lift sections of $L|_Z$. 
If $Z$ is zero dimensional this is easily done, but $Z$ may very well be higher dimensional. Therefore one needs to cut down $Z$ in dimension.
At this stage, two approaches are possible. Helmke's approach is to insist in finding a divisor $D'\in |L|_Z|_{\mathbb{Q}}$ with large order of vanishing at $x$ 
and then lifting $D'$ to $X$. Finding such
$D'$ is now harder than it was finding $D$, as $Z$ may be singular at $x$. However, Helmke proved that if $Z$ is a log canonical center of dimension $d$, then
$\operatorname{mult}_x Z\leqslant \binom{n-1}{n-d}$.
By cutting down one step at a time, one eventually gets a zero dimensional log canonical center. Angehrn and Siu's method is instead to find a divisor $D'\in |L|_Z|_{\mathbb{Q}}$ 
highly singular at a smooth point $y$ near $x$, and then take the limit as $y$ approaches $x$. As we mentioned earlier, both methods give a quadratic bound on $m$.

We will follow Helmke's approach. The crucial new ingredient, however, is to consider all steps simultaneously rather than one at a time. By doing so, we rephrase the problem
of bounding $m$ into an optimization problem of a linear function on a compact convex polyhedron.
This approach allows us to estimate very efficiently the maximum
of the linear function, as it suffices to evaluate it at the vertices of the polyhedron.
We illustrate this idea in the simple case of dimension two. Let $X$ be a smooth surface, let $L$ be an ample line bundle and fix a point $x\in X$. Let
$D\in|L|_{\mathbb{Q}}$ be a $\mathbb{Q}$-divisor such that $\operatorname{ord}_x D\geqslant 1$ and consider the log canonical threshold:
\[
t_1 = \operatorname{sup}\{c|(X, cD) \text{ is log canonical at $x$}\}
\]
Clearly 
\[
0\leqslant t_1\leqslant 2\tag{1}
\]
Let $D_1 = t_1 D$ and let $Z$ be the minimal log canonical center of $(X, D_1)$ at $x$. If $Z$ is zero dimensional, then we are done. If not, then $Z$ is a curve
which is smooth at $x$ by inversion of adjunction. At this point we introduce the following important definition, due to Helmke (see also \cite{ein}).

\begin{definition}
Let $(X,\Delta)$ be a log pair. Let $n$ be the dimension of $X$ and let $x$ be a smooth point in $X$. 
Let $\pi: Y \rightarrow X$ be the blowing up of $X$ at $x$ with exceptional divisor $E$. 
The local discrepancy $b_x(X,\Delta)$ of $(X,\Delta)$ over $x$ is:
\[
 \inf\{b\mid \text{There is a non klt center of $(Y, \pi^* \Delta - (n-1-b)E)$ in $E$}\}
\]
\end{definition}

In our example, let $b_1 = b_x(X, D_1)$. It is then easy to show that
\[
0\leqslant b_1 \leqslant 2- t_1\tag{2}
\]

Also, since $Z$ is one dimensional, we have
\[
0\leqslant b_1\leqslant 1\tag{3}
\]

Next, we want to cut down $Z$. Let $D'\in |L|_Z|_{\mathbb{Q}}$ be a divisor such that $\operatorname{ord}_x D' \geqslant 1$ and let $D''$ be a general lifting of $D'$ to $X$.
Finally, let
\[
t_2 = \operatorname{sup}\{c\mid (X, D_1 + cD'') \text{ is log canonical at $x$}\}
\]
and set $D_2 = D_1 + t_2 D''$.
It is again easy to see that
\[
0\leqslant t_2 \leqslant b_1 \tag{4}
\]

Notice that $x$ is a log canonical center of $(X,D_2)$ and that $D_2\sim_{\mathbb{Q}} (t_1+t_2)L$. It is then a standard argument to deduce that 
$K_X+mL$ has a section that does not vanish at $x$ for $m\geqslant \lceil t_1+t_2+\epsilon \rceil$, so now the problem is to bound $t_1+t_2$.
Consider the set $C\subseteq \mathbb{R}^3$ consisting of points $(t_1, t_2, b_1)$ satisfying conditions $(1)-(4)$ above. 
Then $C$ is the convex hull of the points $(0,0,0)$, $(0,0,1)$, $(2,0,0)$, $(1,0,1)$, $(1,1,1)$, $(0,1,1)$.
Therefore, for any point of $C$ we have $t_1+t_2\leqslant 2$, so that $K_X+mL$ is basepoint free for $m\geqslant 3$.

This same idea applies to higher dimensions. However, the situation is considerably more complicated due to the presence of singularities of $Z$ and due to the fact that the geometry
of $C$ becomes increasingly complex. In order to deal with the problem more efficiently, in higher dimensions we do not compute all the vertices of $C$, but only those for which $\sum_i t_i$ is large.
In the above example, this amounts to noticing that one may rewrite $(2)$ as 
\[
t_1 \leqslant 2-b_1 
\]

By combining this with $(4)$, we get
\[
t_1 + t_2 \leqslant (2-b_1) + b_1\leqslant 2  
\]

The generalization of the expression $(2-b_1)+b_1$ to higher dimensions is the function $f(\underline{b}, \underline{d}, n, 1)$ of Section \ref{sectionoptimization}.
Much of the work of the paper is devoted to carefully estimating $f$ in terms of $n$ only, which then leads to the result. We would like to give here an idea on how this is done at least in the case when $X$ is a threefold.
First, define $D_1$, $b_1$ and $t_1$ as above and suppose that $\operatorname{LLC}(X, D_1, x)=\{Z_1\}$, where $Z_1$ is an irreducible surface. By Theorem \ref{multiplicity} we have that
\[
m_1 = \operatorname{mult}_x Z_1 \leqslant 3-\lceil b_1 \rceil 
\]

For the next step, choose a $\mathbb{Q}$-divisor $D'\in |L|_{Z_1}|_{\mathbb{Q}}$ with large order of vanishing at $x$. Notice that, differently than above,
the best bound we may hope for is $\operatorname{ord}_x D' \geqslant \frac{1}{\sqrt{m_1}}$, due to the fact that $Z_1$ is possibly singular (see Definition \ref{multdef} for 
the definition of order of vanishing in this context). Let $D''$ be a general lifting of $D'$ to $X$, let 
\[
t_2 = \operatorname{sup}\{c\mid (X, D_1 + cD'') \text{ is log canonical at $x$}\}
\]
and set $D_2 = D_1 + t_2 D''$. If we define $b_2 = b_x(X, D_2)$, then by an argument due to Helmke (see also Theorem \ref{cutlc}) we have that
\[
b_2 \leqslant b_1 - t_2 \cdot \operatorname{ord}_x D' \leqslant b_1 - \frac{t_2}{\sqrt{m_1}}
\]
Now suppose that $\operatorname{LLC}(X, D_2, x)=\{Z_2\}$, where $Z_2$ is an irreducible curve. 
Then we have that $Z_2$ is smooth near $x$ by inversion of adjunction. 
Let $D'\in |L|_{Z_2}|_{\mathbb{Q}}$ be a divisor such that $\operatorname{ord}_x D' \geqslant 1$ and let $D''$ be a general lifting of $D'$ to $X$.
Finally, let
\[
t_3= \operatorname{sup}\{c\mid (X, D_2 + cD'') \text{ is log canonical at $x$}\}
\]
and set $D_3 = D_2 + t_3 D''$.
We have that $t_3 \leqslant b_2$ and that $x$ is a log canonical center of $(X,D_3)$. Putting everything together, we get
\begin{equation*}
t_1 + t_2 + t_3\leqslant 3-b_1 + (b_1-b_2)\sqrt{3-\lceil b_1 \rceil} + b_2
\end{equation*}

Notice that $b_1 \leqslant \operatorname{dim}(Z_1) = 2$. Therefore, $\lceil b_1 \rceil \neq 3$ and the above expression does not decrease if we decrease $b_2$. In particular, we may assume $b_2=0$.
But then we get
\[
t_1 + t_2 + t_3 \leqslant 3 - b_1 + b_1\sqrt{3-\lceil b_1 \rceil}\leqslant 2+\sqrt{2}<4
\]
This proves Fujita's basepoint freeness conjecture in dimension three, at least in the case when the log canonical centers constructed inductively have dimension two and one respectively (the other cases being entirely
analogous). 
As $n$ grows larger, however, bounding $f(\underline{b}, \underline{d}, n, 1)$ becomes increasingly difficult. For example, if $X$ is a fourfold and if there are four steps in the inductive process, the upper
bound on $\sum_i t_i$ is
\[
4-b_1 + (b_1-b_2)\sqrt[3]{4-\lceil b_1 \rceil} + (b_2-b_3)\sqrt{\binom{4-\lceil b_2 \rceil}{2}}+ b_3
\]

We refer to Section \ref{sectionoptimization} for the details of the estimates on $f(\underline{b}, \underline{d}, n, r)$ and we refer to Appendix \ref{appendix} for a proof of their optimality. 

\ \

\noindent \textbf{Acknowledgements:} We would like to thank Prof. James M\textsuperscript{c}Kernan and Prof. Bangere Purnaprajna for their encouragement and for helpful comments on the paper. We would also like to thank Fei Ye and Zhixian Zhu for interesting discussions on Section \ref{sectionsixfolds}.
The research of LG is supported by the Alexander von Humboldt Research Fellowship for Postdoctoral Researchers.

\section{Preliminaries}

\subsection{Notation}

We work over an algebraically closed field $k$ of characteristic zero.  
Most of the following notation is standard. 
$\mathbb{N}$ is the set of natural numbers, zero included. We denote the logarithmic function with natural base as $\operatorname{log}:\mathbb{R}^+ \rightarrow \mathbb{R}$. We denote by $W\colon \mathbb R_{\geq 0} \to \mathbb R_{\geq 0}$ the principal branch of Lambert's productlog function (see Definition \ref{definitionW}). 
A $\mathbb{Q}$-Cartier divisor $D$ on a normal variety $X$ is nef if $D\cdot C\geqslant 0$ for any curve $C\subseteq X$.
We use the symbol $\sim_\mathbb{Q}$ to indicate $\mathbb{Q}$-linear equivalence and the symbol $\equiv$ to indicate numerical equivalence.
We denote by $|D|_\mathbb{Q}$ the $\mathbb{Q}$-linear series of a $\mathbb{Q}$-Cartier divisor $D$.
A pair $(X,\Delta)$ consists of a normal variety $X$ and a $\mathbb{Q}$-Weil divisor $\Delta$ such that $K_X + \Delta$ is $\mathbb{Q}$-Cartier.
If $\Delta\geqslant 0$, we say $(X,\Delta)$ is a log pair. If $f:Y \rightarrow X$ is a birational morphism, we may write
$K_Y + f^{-1}_* \Delta = f^*(K_X+ \Delta) + \sum_i a_i E_i$ with $E_i$ $f$-exceptional divisors. A log pair $(X,\Delta)$ is called log canonical (or lc) if 
$a_i\geqslant -1$ for every $i$ and for every $f$, and it's called Kawamata log terminal (or klt) if $a_i>-1$ for every $i$ and $f$, and furthermore 
$\lfloor\Delta \rfloor=0$. The rational numbers $a_i$ are called the discrepancies of $E_i$ with respect to $(X,\Delta)$ and do not depend on $f$.
We say that a subvariety $V\subseteq X$ is a non klt center if it is the image of a divisor of discrepancy at most $-1$. 
A non klt center $V$ is a log canonical center if $(X,\Delta)$ is log canonical at the generic point of $V$.
A non klt place (respectively log canonical place) is a valuation corresponding to a divisor of discrepancy at most (respectively equal to) $-1$. 
The set of all log canonical centers passing though $x\in X$ is denoted by $\operatorname{LLC}(X,\Delta,x)$,
and the union of all the non klt centers is denoted by $\operatorname{Nklt}(X,\Delta,x)$.
Finally, the log canonical threshold of $(X,\Delta)$ at a point $x$ is 
$\operatorname{lct}(X,\Delta,x)=\operatorname{sup}\{c>0| (X,c\Delta) \text{ is lc at $x$}\}$.

\subsection{Log canonical centers}

We recall here some standard definitions and results in birational geometry for the convenience of the reader. 

\begin{definition}
Let $X$ be an irreducible projective variety of dimension $n$ and let $D$ be a $\mathbb{Q}$-Cartier divisor. 
Let $m$ be a positive integer such that $mD$ is Cartier. The volume of $D$ is:
\[
   \operatorname{vol}(X,D)=\limsup_{k\to\infty}\frac{n! h^0(X,kmD)}{(km)^n}.
\]
\end{definition}

\begin{definition}\label{multdef}
Let $X$ be an irreducible projective variety, let $x$ be a point of $X$ and let $D$ be a $\mathbb{Q}$-Cartier divisor on $X$. 
Let $m$ be a positive integer such that $mD$ is Cartier and let $f\in \mathcal{O}_{X,x}$ be a defining equation.
Then we define the order of vanishing of $D$ at $x$ as
\[
\operatorname{ord}_x D = \frac{1}{m}\operatorname{max}\{s\in \mathbb{N}| f\in \mathfrak{m}_x ^s\}
\]
\end{definition}

\begin{lemma}\label{volumesections}
Let $X$ be an irreducible projective variety of dimension $n$ and let $D$ be a $\mathbb{Q}$-Cartier divisor on $X$. 
Let $T$ be a finite set of points of $X$ of cardinality $r$.
Then there is a $\mathbb{Q}$-divisor $D'\in |D|_{\mathbb{Q}}$ such that 
\[
\operatorname{ord}_x D' \geqslant \Big( \frac{\operatorname{vol}(X, D)}{r\operatorname{mult}_x X} \Big)^{1/n}
\]
for all $x\in T$. 
\end{lemma}
\begin{proof}
See \cite[Proposition 3.2]{helmke1} or \cite[Proposition 2.1]{kawamata1}.
\end{proof}

\begin{definition}
Let $(X,\Delta)$ be a log pair with $X$ smooth, and let $\mu : Y\rightarrow X$ be a log resolution. We define the multiplier ideal sheaf of the pair $(X,\Delta)$ to be
\[
   \mathcal{I}(X,\Delta) = \mu _* \mathcal{O}_Y (K_{Y/X} - \lfloor \mu ^* \Delta \rfloor) \subseteq \mathcal{O}_X
\]
\end{definition}

We have that $(X,\Delta)$ is klt if and only if $\mathcal{I}(X,\Delta)=\mathcal{O}_X$, and
it is lc if and only if $\mathcal{I}(X, (1-\epsilon)\Delta)=\mathcal{O}_X$ for any $0<\epsilon\ll 1$. Therefore
$\operatorname{Nklt}(X,\Delta)=\operatorname{Supp}(\mathcal{O}_X / \mathcal{I}(X,\Delta))$. 

\begin{theorem}[Nadel vanishing theorem]\label{nadel}
Let $X$ be a smooth projective variety and $\Delta\geqslant 0$ a $\mathbb{Q}$-divisor on $X$. Let $D$ be any integral divisor such that $D-\Delta$ is big 
and nef. Then $H^i (X, \mathcal{O}_X (K_X + D)\otimes \mathcal{I}(X,\Delta))=0$ for $i>0$.
\end{theorem}
\begin{proof}
See \cite[Section 9.4.B]{lazarsfeld2}.
\end{proof}

\begin{proposition}\label{mult}
Let  $X$ be an irreducible variety of dimension $n$ and let $\Delta$ be an effective $\mathbb{Q}$-divisor on $X$. If $\operatorname{ord}_x \Delta\geqslant n$ at some
smooth point $x\in X$, then $\mathcal{I}(X,\Delta)_x\subseteq \mathfrak{m}_x$, where $\mathfrak{m}_x$ is the maximal ideal of $x$.
\end{proposition}
\begin{proof}
This is \cite[Proposition 9.3.2]{lazarsfeld2}.
\end{proof}

\begin{lemma}\label{mlc}
Let $(X,\Delta)$ be a log pair such that $\Delta$ is $\mathbb{Q}$-Cartier. Assume that $X$ is klt and $(X,\Delta)$ is lc. If $W_1$ and $W_2$ are log canonical centers of $(X,\Delta)$ 
and $W$ is an irreducible component of $W_1 \cap W_2$, then $W$ also is a log canonical center of $(X,\Delta)$. In particular, if $(X,\Delta)$ is not klt at $x\in X$,
then there exists the unique minimal element of $\operatorname{LLC}(X,\Delta,x)$.
\end{lemma}
\begin{proof}
See \cite[Proposition 1.5]{kawamata1}.
\end{proof}

We will refer to the following result as \say{tie breaking}. 

\begin{lemma}\label{tiebreak}
Let $(X,\Delta)$ be a log pair such that $X$ is klt and $\Delta$ is $\mathbb{Q}$-Cartier. Let $S$ be a finite set of points of $X$. Suppose 
that there is a point $x\in S$ such that $\{x\}\in \operatorname{LLC}(X,\Delta,x)$ and that for each point $y\in S\setminus\{x\}$
there is a non klt center of $(X,\Delta)$ containing $y$ but not $x$. Let $D$ be an ample 
$\mathbb{Q}$-Cartier divisor. Then there exists a positive rational number $a>0$ such that for any $0<\epsilon\ll 1$ there exists
a $\mathbb{Q}$-Cartier divisor $E\in |aD|_{\mathbb{Q}}$ such that
\begin{enumerate}
\item $(X,(1-\epsilon)\Delta + \epsilon E)$ is not klt at any point of $S$.
\item $(X,(1-\epsilon)\Delta + \epsilon E)$ is lc at $x$.
\item $\operatorname{LLC}(X,(1-\epsilon)\Delta + \epsilon E, x)=\{x\}$. 
\end{enumerate}
\end{lemma}
\begin{proof}
This is an analogue of \cite[Proposition 6.2]{helmke1}. 
For each $y\in S\setminus\{x\}$ let $v_y$ be a non klt place of $(X,\Delta)$ whose center $Z_y$ contains $y$ but not $x$.
Let $b$ be a positive rational number such that there exists $E_1\in|bD|_{\mathbb{Q}}$ with $v_y(E_1) >  v_y(\Delta)$
for all $y\in S\setminus\{x\}$.
After possibly taking a larger $b$ we may assume that the common support of all such $E_1$ is exactly the union of the $Z_y$.
Similarly, let $v_x$ be a log canonical place of $(X,\Delta)$ with center $x$ and
let $c$ be a positive rational number such that there exists $E_2\in |cD|_{\mathbb{Q}}$ with $v_x(E_2) > v_x(\Delta)$.
Again, after possibly taking a larger $c$ we may assume that the common support of all such $E_2$ is exactly $x$. Set $a=b+c$.
For general choices of $E_1$ and $E_2$ the pair $(X,(1-\epsilon)\Delta + \epsilon E_1 + \epsilon E_2)$ is not lc at any point of $S$ for any small $\epsilon$. Let 
\[
t=\operatorname{sup}\{d|(X,(1-\epsilon)\Delta + \epsilon E_1 + d\epsilon E_2) \text{ is lc at $x$}\}
\]
Clearly $t<1$. Finally, take $E_3\in |(1-t)c D|_{\mathbb{Q}}$ general enough. We have
\[
\operatorname{LLC}(X, (1-\epsilon)\Delta + \epsilon (E_1 + tE_2 + E_3))=\{x\}
\]
Furthermore, $S$ is contained in $\operatorname{Nklt}(X, (1-\epsilon)\Delta + \epsilon (E_1 + tE_2 + E_3))$
and $E_1 + tE_2 + E_3\in |aD|_{\mathbb{Q}}$. Therefore, we may take $E=E_1 + tE_2 + E_3$.
\end{proof}

\section{The inductive method}\label{sectioninduction}

In this section we describe an inductive method for cutting down the dimension of non klt centers. This essentially due to Helmke (see in particular \cite[Proposition 6.3]{helmke1}).
Since we will need Helmke's result in a slightly different form, we go over its proof and make the appropriate changes. 

\begin{proposition}\label{cutlc} 
Let $(X,\Delta)$ be a log pair, where $X$ is a smooth projective variety of dimension $n$. 
Let $S$ be a finite set of points contained in $\operatorname{Nklt}(X,\Delta)$ and let $r$ be the cardinality of $S$. Let $T$ be a nonempty subset of $S$ such that:
\begin{enumerate}
\item $(X,\Delta)$ is log canonical at all points of $T$.
\item All points of $T$ share a common minimal log canonical center $Z$. 
\item Every point in $S\setminus T$ is contained in a non klt center of $(X,\Delta)$ that does not contain any point of $T$. 
\end{enumerate}
Let $d=\operatorname{dim}Z$ and let $D$ be an ample $\mathbb{Q}$-divisor.
If $d>0$, then there exists a nonempty subset $T'$ of $T$, a rational number $t$ such that
\[
0\leqslant t \leqslant b_x(X,\Delta) \Big( \frac{r\operatorname{mult}_x Z}{D^d \cdot Z}\Big) ^{1/d} 
\]
for every point $x\in T$ and a $\mathbb{Q}$-divisor $D'\in |D|_{\mathbb{Q}}$ such that 
\begin{enumerate}
\item $(X,\Delta+tD')$ is log canonical at all points of $T'$.
\item All points of $T'$ share a common minimal log canonical center $Z'$ strictly contained in $Z$.
\item Every point in $S\setminus T'$ is contained in a non klt center of $(X,\Delta+tD')$ that does not contain any point of $T'$. 
\end{enumerate}
Furthermore
\[
b_x (X,\Delta+tD')\leqslant b_x(X,\Delta) - t \cdot \Big( \frac{D^d\cdot Z}{r\operatorname{mult}_x Z}\Big)^{1/d}
\]
for all points $x\in T'$.
\end{proposition}
\begin{proof}
Notice that $D^d \cdot Z=\operatorname{vol}(Z, D|_Z)$.
Then by Lemma \ref{volumesections} there exists a $\mathbb{Q}$-divisor $D''\in |D|_Z|_{\mathbb{Q}}$ such that 
\[
\operatorname{ord}_x D'' \geqslant \Big( \frac{D^d\cdot Z}{r\operatorname{mult}_x Z}\Big)^{1/d}
\]
for every point $x$ in $T$.
Let $D'\in|D|_{\mathbb{Q}}$ be a general lifting of $D''$ to $X$ and let 
\[
t=\sup\{c| (X,\Delta+cD') \text{ is log canonical at some point of $T$}\}
\]
By the definition of local discrepancy over $x$ and by the proof of \cite[Proposition 3.2]{helmke1}, we have 
\[
0\leqslant t \leqslant b_x(X,\Delta) \Big( \frac{r\operatorname{mult}_x Z}{D^d \cdot Z}\Big) ^{1/d}
\]
for all points $x$ in $T$.
Let $T_1$ be the set of points of $T$ where $(X,\Delta+tD')$ is log canonical.
For each $x$ in $T_1$ let $Z_x$ be the minimal log canonical center of $(X,\Delta+tD')$ at $x$. By construction, for any $x\in T_1$ we have that
$Z_x$ is strictly contained in $Z$. Choose a maximal element $Z'$ in the set
$\{Z_x | x\in T_1\}$ ordered by inclusion and let $T'=\{x\in T_1| Z_x = Z'\}$. Now, if $x\in S\setminus T$ then there is a non klt center of $(X,\Delta+tD')$
containing $x$ but none of the points of $T'$ by hypothesis. If $x\in T\setminus T'$ either $(X,\Delta+tD')$ is not log canonical at $x$, or $(X,\Delta+tD')$ is log canonical at $x$
but the minimal log canonical center $Z_x$ does not contain $Z'$.  In either case, if $x\in S\setminus T$ there is a non klt center which does not contain any of the points of $T'$.

The final statement of the Proposition is also clear.
\end{proof}

\begin{remark}\label{fulldim}
If $\Delta=0$ in Proposition \ref{cutlc}, then the conclusion holds without any hypothesis. 
\end{remark}

Proposition \ref{cutlc} shows that it is crucial to have control over the singularities of log canonical centers. In this direction,
we have:

\begin{theorem}\label{multiplicity}
Let $X$ be a smooth projective variety and $(X,\Delta)$ be log canonical at $x\in X$. Let $Z_d$ be the union 
of the elements of $\operatorname{LLC}(X,\Delta,x)$ of dimension $d$. Then
\[
\operatorname{mult}_x Z_d \leqslant \binom{n-\lceil b_x(X,\Delta)\rceil}{n-d}
\]
\end{theorem}
\begin{proof}
See \cite[Theorem 4.3]{helmke1}. 
\end{proof}

\section{Optimization}\label{sectionoptimization}

Let $s<n$ be nonnegative integers and let $r$ be any positive integer. Consider the set $R_{s,n} \subseteq \mathbb{R}^{s+2}\times \mathbb{N}^{s+2}$ consisting of elements
\[
(\underline{b}, \underline{d}) = (b_0, b_1, \cdots, b_s, b_{s+1}, d_0, d_1, \cdots, d_s, d_{s+1})
\] 
satisfying the following conditions:
\[
0= b_{s+1}  < b_s < \cdots < b_1 < b_0=n
\]
\[
0= d_{s+1} < d_s < \cdots < d_1<d_0=n
\]
and $b_i \leqslant d_i$ for all $1\leqslant i\leqslant s$. 
This section is devoted to the study of the functions:
\[
f(\underline{b}, \underline{d}, n, r)=\sum_{i=0} ^ s (b_i - b_{i+1}) \left[ r\binom{n-\lceil b_i\rceil}{n-d_i} \right]^{1/d_i}
\]
and
\[
F(n,r)=\operatorname{max}\{f(\underline{b}, \underline{d}, n, r)\mid  (\underline{b}, \underline{d})\in \cup_{s=0} ^{n-1} R_{s,n}\}
\]

In particular, we aim to prove the following upper bounds.

\begin{theorem}\label{upperbound}
	Let $n$ and $r$ be positive integers. Then:
	\begin{enumerate}
		\item $F(n,1)<\operatorname{max}\{n+1, n(\operatorname{log}\operatorname{log}(n)+2.34)\}$
		\item $F(n,r) < r+n-1 + \sqrt{r} n (\operatorname{log}\operatorname{log}(n)+2.34)\}$
	\end{enumerate}
\end{theorem}

We start by pointing out that in order to maximize $f$, it is sufficient to consider integral values of $b_i$. 

\begin{lemma}\label{maxinteger}
	Let $(\underline{b}, \underline{d})\in R_{s,n}$. Then there is an integer $s'\leqslant s$ and an element $(\underline{b}', \underline{d}')\in R_{s', n}$ such that
	the vector $\underline{b}'$ consists of integers and $f(\underline{b}, \underline{d}, n, r)\leqslant f(\underline{b}', \underline{d}', n, r)$.
\end{lemma}

\begin{proof}

	Consider the set  $B_s\subseteq \mathbb R^{s+2}$ consisting of elements
	$$ \underline b' = (b'_0,\cdots, b'_{s+1} )  $$
	satisfying the conditions 
	\begin{gather*}
		0= b'_{s+1}\leqslant b'_s \leqslant \cdots \leqslant b'_1 \leqslant b'_0=n, \\
		\lceil b_i\rceil -1 \leqslant b'_i \leqslant \lceil b_i \rceil
	\end{gather*}
	Consider the linear function: 
	\[
	L(\underline{b}') = \sum_{i=0} ^ s (b'_i - b'_{i+1}) \left[ r\binom{n-\lceil b_i\rceil}{n-d_i} \right]^{1/d_i}
	\]
         Notice that if $\underline{b'}$ is in the interior of $B_s$ then we have that $f(\underline{b'}, \underline{d}, n, r)=L(\underline{b'})$. 
         Also, $f(\underline{b}, \underline{d}, n, r)=L(\underline{b})$. 
	The set $B_{s}$ is a convex compact subset of $\mathbb R^{s+2}$, therefore $L$ achieves its maximum value at a vertex $\underline b'\in B_s$. 
	By construction, all the vertices of $B_s$ have integral coordinates. We have:
	$$ f(\underline{b}, \underline{d}, n, r) = L(\underline b)\leqslant L(\underline b')\leqslant f(\underline b', \underline d, n, r).$$
	
	After possibly erasing the entries of $b_i'$ which are repeated and the corresponding $d_i'$, we may assume that $(\underline{b'}, \underline{d'})$ belongs to $R_{s',n}$ for some $s'\leqslant s$. 
\end{proof}

Notice that if $n$ and $r$ are fixed, Lemma \ref{maxinteger} 
reduces the computation of $F(n,r)$ to finitely many steps. If $n$ is small enough, this computation may be carried out by a computer.
We list in Table 1 the first few values of $\lfloor F(n,r) \rfloor$. 

\begin{table}[h!] \label{tablefirstvalues}
	\begin{center}
\begin{tabular}
	{c|| cccccccccccccccc}
	
\diaghead(3,-2){fjjak}  {\,$r$} {$n$} & \ 2  & 3 & 4 & 5 & 6 & 7 & 8 & 9 & 10 & 11 & 12 & 13 & 14 & 15 & 16 & 17\\
\midrule
1& \ 2 & 3 & 4 & 6 & 8 & 9 & 11&13&15 & 17 & 19 & 21 & 24 & 26 & 28 &30\\
2&\  3 & 4 & 6 & 8 &10&11&13&15 &18 & 20 & 22 & 24 & 26 & 28 & 30 & 33\\
\bottomrule
\end{tabular}

\end{center}
\vspace{2pt}
\caption {Values of $\lfloor F(n,r) \rfloor$ for $2\leqslant n \leqslant 17$ and $r=1,2$.}
\end{table}

Next, by dropping the condition that $d_{j+1}$ is strictly less than $d_{j}$ for all $j$, we show that we may reduce to the case in which  $\underline b$ is the sequence of natural numbers 
$(0,1,\cdots, n)$
ranging from 0 to $n$.

\begin{lemma}\label{sumbound}
	Let $(\underline{b}, \underline{d})\in R_{s,n}$. Then there is a function $d:\{1,\cdots ,n\}\to \mathbb N$ with $b\leq d(b)\leq n$ such that:
	
	\[
	f(\underline{b},\underline{d},n,r) 
	\leqslant 
	\sum_{b=1} ^n \left[ r\binom{n-b}{n-d(b)}\right] ^{1/d(b)}
	\]
	
\end{lemma}
\begin{proof}
	By Lemma \ref{maxinteger}, we may assume that all the $b_i$ are integers. 
	Now if $x\leqslant n$ is a positive real number, we define $i(x)$ by the property $b_{i(x)+1} < x \leqslant b_{i(x)}$ and set $d(x)=d_{i(x)}$.
	Therefore, we have:
	
	\begin{equation*}
		\begin{split}
			f(\underline{b}, \underline{d}, n, r)&=\sum_{i=0} ^ {s} (b_i - b_{i+1}) \left[ r\binom{n-b_i}{n-d_i} \right]^{1/d_i}\\
			&= \sum_{i=0} ^{s} \sum_{b=b_{i+1}+1} ^{b_i} \left[ r\binom{n-b_i}{n-d_i} \right]^{1/d_i}\\
			&\leqslant\sum_{i=0} ^{s} \sum_{b=b_{i+1}+1} ^{b_i} \left[ r\binom{n-b}{n-d(b)} \right]^{1/d(b)}\\
			&=\sum_{b=1} ^n \left[ r\binom{n-b}{n-d(b)}\right] ^{1/d(b)}
		\end{split}
	\end{equation*}	
\end{proof}

We now need to measure the contribution of each term in Lemma \ref{sumbound}. 
We start with the following elementary estimate, which is not optimal, but already implies an upper bound on $F(n,r)$ that is quadratic in $n$ and essentially linear in $r$.

\begin{lemma}\label{elementaryestimate}
	Let $b\leq d\leq n $ and $r$ be positive integers. Then 
	$$\left[ r\binom{n-b}{n-d}\right] ^{1/d} \leq \sqrt[b]r  + n-b $$
\end{lemma}

\begin{proof}
First, we have that $\binom{n-b}{n-d}\leq n^{d-b}$. Then, using Young's inequality  $A^\lambda B^{1-\lambda} \leq \lambda A + (1-\lambda) B$ we get
\begin{align*}
	r^{1/d} \binom{n-b}{n-d}^{1/d} 
	&\leq (\sqrt[b] r)^{b/d} n^{1- b /d}\\
	& \leq  \frac b d \sqrt[b] r + \left(1 - \frac b d\right) n\\
	& \leq  \frac b b \sqrt[b] r + \left(1 - \frac b n\right) n\\
	& = \sqrt[b]r  + n-b
\end{align*}
\end{proof}

\begin{corollary}\label{elementarycorollary} For all positive integers $n$ and $r$ we have
	$$F(n,r) \leq \frac {n(n-1)} 2 + \sum _{b=1}^n \sqrt[b]r $$
\end{corollary}

\begin{proof}
	This follows from Lemma \ref{sumbound} and Lemma \ref{elementaryestimate}.
\end{proof}

\subsection{Optimization with Lambert's W function}

In order to give sharper estimates on $F(n,r)$ for large values of $n$, it is convenient to introduce the Lambert function.

\begin{definition}\label{definitionW}
Consider the function $u:\mathbb{R}_{\geqslant 0}\rightarrow \mathbb{R}_{\geqslant 0}$ defined by $u(x)=xe^x$. We define the Lambert function $W:\mathbb{R}_{\geqslant 0}\rightarrow \mathbb{R}_{\geqslant 0}$
as the inverse of $u$. 
\end{definition}

\begin{lemma}\label{Wbound}
Let $b\leqslant d\leqslant n$ and $r$ be positive integers. Then
\[
\left[ r\binom{n-b}{n-d}\right] ^{1/d}\leqslant \sqrt[b]{r} \operatorname{exp}\left( W\left( \frac{n}{b \sqrt[b]{r}} \right) \right)
\]
\end{lemma}
\begin{proof}
If $b=d$ there is nothing to prove, because the binomial on the left-hand side reduces to 1, while the exponential on the right-hand side is $\geqslant 1$. Therefore, suppose that $d>b$.

By the basic version of Stirling's inequality $A!\geqslant (A/e)^A$ we obtain:
\[
\binom{n-b}{n-d} =  \frac{(n-b)\cdot (n-b-1) \cdots (n-d+1)}{(d-b)!}\leqslant \left( \frac{en}{d-b} \right) ^{d-b}
\]
Therefore:
\[
\binom{n-b}{n-d}^{1/d} \leqslant \left( \frac{en}{d-b} \right) ^{\frac{d-b}{d}}
\]

Let $\delta=d-b$. By taking the logarithmic derivative in $\delta$, we see that the expression
\[
\psi(\delta)=r^{\frac{1}{b+\delta}} \left( \frac{en}{\delta} \right) ^{\frac{\delta}{b+\delta}}
\]

is maximized when 
\[
\delta + \operatorname{log}(r) = b(\operatorname{log}(n)-\operatorname{log}(\delta))
\]
Then $\delta$ may be expressed as $\delta=b w $ where $w = W\big( \frac{n}{b\sqrt[b]{r}}\big)$. %here big( seems to be working much better than \left(
Now, plugging this value of $\delta$ into $\psi$, and using the defining properties of $W$, we get:
\[
r^{\frac{1}{b+\delta}} \left( \frac{en}{\delta} \right) ^{\frac{\delta}{b+\delta}}
 =
  \sqrt[b]{r}^{\frac 1 {1+w}}  \left(e \sqrt [b]r e^w \right)^{\frac w {w+1}} = \sqrt [b]r e^w
\]
\end{proof}
\begin{remark}\label{remarkWbound}
	By the properties of the Lambert function, the previous lemma may also be written in the following way:
	$$ \left[ r\binom{n-b}{n-d}\right] ^{1/d} \leq  n/\delta (b) $$
	where $\delta(b) = b W\big( n/({b\sqrt[b]{r}})\big)$. 
\end{remark}

We are ready to prove Theorem \ref{upperbound}. We start with its first part.

\begin{theorem}\label{upperbound1}
Let $n$ be a positive integer. Then: 
\[
F(n,1)<\operatorname{max}\{n+1, n(\operatorname{log}\operatorname{log}(n)+2.34)\}
\]
\end{theorem}
\begin{proof}
By Lemma \ref{sumbound} and Lemma \ref{Wbound}, we get:
\[
f(\underline{b}, \underline{d}, n, 1) \leqslant \sum_{b=1} ^n  e^{W(n/b)}\leqslant e^{W(n)} + \int_1 ^n e^{W(n/x)}dx
\]

By the change of variable $t=n/b$ we get:

\[
\int_1 ^n e^{W(n/x)}dx = n\int_1 ^n \frac{e^{W(t)}}{t^2} dt = n\left[\operatorname{log}W(t)- \frac{1}{W(t)}\right]_1 ^n 
\]

Therefore
\begin{equation*}
\begin{split}
f(\underline{b}, \underline{d}, n, 1) &\leqslant n\left (\frac 1 {W(n)}+ \operatorname{log}W(n)-\frac 1 {W(n)}-\operatorname{log}W(1)+\frac 1 {W(1)}\right)\\
&= n \ \big(\operatorname{log}W(n)-\operatorname{log}W(1)+W(1)^{-1}\big)
\end{split}
\end{equation*}

Now, for $n\geqslant 3$ we have that $W(n)\leq \log n$ and so
\[
\operatorname{log}W(n)-\operatorname{log}W(1)+1/W(1)< \operatorname{log}\operatorname{log}(n)+2.34
\]

For $n<3$, we may use Table 1 instead.
\end{proof}

Similarly, we prove now the second part of Theorem \ref{upperbound}.

\begin{theorem}\label{upperbound2}
Let $n,r\geq 2$ be a positive integers. Then: 
\[
F(n,r)<r+n-1+\sqrt r  n(\operatorname{log}\operatorname{log}(n)+2.34)
\]
\end{theorem}
\begin{proof}
As in the case $r=1$, we start with Lemma \ref{sumbound}, which gives
\[
f(\underline{b}, \underline{d}, n, r) \leqslant \sum_{b=1} ^n
\left[ r\binom{n-b}{n-d}\right] ^{1/d}
\]
This time, however, we estimate with Lemma \ref{Wbound} only the terms of the sum with $b\geqslant 2$. For the first term, instead, we use Lemma \ref{elementaryestimate}. We get

\[
f(\underline{b}, \underline{d}, n, r)
\leqslant
 r+ n-1 +\sum_{b=2} ^n  \sqrt[b]r  e^{W(n/(b\sqrt[b]r))}
\]

Since $W(n/(b\sqrt[b]r)\leq W(n/b)$ and $\sqrt[b]r\leq \sqrt r$ for $b\geq 2$, we may continue as in the proof of Theorem \ref{upperbound1}:

\begin{align*}
	f(\underline{b}, \underline{d}, n, r)
	&\leqslant 
	r+n-1	+ \sqrt r\, \int_1 ^n e^{W(n/x)}dx\\
	& \leq 
	r+n-1 + \sqrt r \, n (\log W(n) - W(n)^{-1} + 2.34)\\
	& \leq 
	r+n-1 + \sqrt r\,  n (\log \log n + 2.34)
\end{align*}
\end{proof}

\section{The main result}\label{sectionmainresult}

Here we apply the methods developed so far to prove Theorem \ref{intro1} and Theorem \ref{intro3}. We start with the following:

\begin{theorem}\label{main}
Let $X$ be a smooth projective variety of dimension $n$ and let $D$ be an ample $\mathbb{Q}$-divisor. Let $S$ be any finite set of points of $X$ of cardinality $r$.
Suppose that 
\[
D^d \cdot Z\geqslant 1
\]

for all irreducible $d$-dimensional subvarieties $Z$ containing at least one point of $S$. 
Fix any positive rational number $0<\epsilon\ll 1$. 
Then there exists a point $x\in S$ and a 
$\mathbb{Q}$-divisor $\Delta\in |tD|_{\mathbb{Q}}$ such that:
\begin{enumerate}
\item $t < F(n,r)+\epsilon$.
\item $(X,\Delta)$ is log canonical but not Kawamata log terminal at $x$.
\item $\operatorname{LLC}(X,\Delta,x)=\{x\}$.
\item $S$ is contained in $\operatorname{Nklt}(X,\Delta)$. 
\end{enumerate}
\end{theorem}
\begin{proof}
We define inductively a sequence of $\mathbb{Q}$-divisors, subvarieties, finite sets of points, positive rational numbers and positive integers
$(D_i, Z_i, T_i, t_i, d_i)$ for $0\leqslant i\leqslant s+1$ as follows. 
Set $D_0 = 0$, $Z_0 = X$, $T_0 = S$, $t_0 = 0$ and $d_0 = n$.
Now suppose that we are given $(D_i, Z_i, T_i, t_i, d_i)$. If $i>0$, suppose that:
\begin{enumerate}
\item $(X, D_i)$ is log canonical at all points of $T_i$.
\item All points of $T_i$ share a common minimal log canonical center $Z_i$.
\item Every point in $S\setminus T_i$ is contained in a non klt center of $(X, D_i)$ that does not contain any point of $T_i$.
\item $d_i  = \operatorname{dim}(Z_i)$.
\end{enumerate}

If $d_i = 0$, we stop. If not, we construct $(D_{i+1}, Z_{i+1}, T_{i+1}, t_{i+1}, d_{i+1})$ as follows. 
By Proposition \ref{cutlc} (see also Remark \ref{fulldim} for the case $i=0$), there exists
a nonempty subset $T'$ of $T_i$, a rational number $t$ and a $\mathbb{Q}$-divisor $D'\in |D|_{\mathbb{Q}}$ such that 
\begin{enumerate}
\item $(X, D_i + tD')$ is log canonical at all points of $T'$.
\item All points of $T'$ share a common minimal log canonical center $Z'$ strictly contained in $Z_i$. 
\item Every point in $S\setminus T'$ is contained in a non klt center of $(X, D_i + tD')$ that does not contain any point of $T'$.
\item For all points $x\in T'$ we have
\[
b_x(X,D_i+tD')\leqslant b_x(X,D_i) - \frac{t}{(r\operatorname{mult}_x Z_i)^{1/d_i}}
\]
\end{enumerate}

We set $D_{i+1} = D_i + tD'$, $Z_{i+1}=Z'$, $T_{i+1}=T'$, $t_{i+1}=t$ and $d_{i+1}=\operatorname{dim}(Z')$.
By construction, $Z_{s+1}$ is zero dimensional and non-empty. Let $x$ be any point contained in $Z_{s+1}$. We define
a sequence of positive rational numbers and positive integers $(b_i, m_i)$ for $0\leqslant i\leqslant s+1$ as follows. We set
$b_0 = n$ and $m_0=1$. For any $i>0$, we set
$b_i = b_x(X, D_i)$ and $m_i =\operatorname{mult}_x Z_i$. By $(4)$ above we have that for all $0\leqslant i\leqslant s$:
\[
t_{i+1} \leqslant (b_i-b_{i+1})\cdot (rm_i)^{1/d_i}\tag{\text{*}}
\]

By Theorem \ref{multiplicity} we have that
\[
m_i \leqslant \binom{n-\lceil b_i\rceil}{n-d_i}
\]

Therefore by $(*)$:
\[
\sum_{i=0} ^s t_{i+1}\leqslant f(\underline{b}, \underline{d}, n, r) \leqslant F(n,r)
\]

We may now conclude by tie breaking with Lemma \ref{tiebreak}.
\end{proof}

\begin{theorem}\label{gg}
Let $X$ be a smooth projective variety of dimension $n$, let $L$ be an ample line bundle and let
$r$ be a positive integer. Then $K_X + mL$ separates $r$ points for $m > F(n,r)$.
\end{theorem}
\begin{proof}
Since $L$ is a line bundle we have that $L^d \cdot Z\geqslant 1$ for all irreducible $d$-dimensional subvarieties $Z$. Let $m>F(n,r)$ be any positive integer. 
We prove that $K_X + mL$ separates $r$ points by induction on $r$. The first step is to show that $K_X + mL$ is base point free.
Let $x$ be any point of $X$ and let $t$ and $\Delta\in|L|_{\mathbb{Q}}$ be as in Theorem \ref{main} with $S=\{x\}$. Then $t < F(n,r)+\epsilon$ for a small positive rational number $\epsilon$. 
Consider the short exact sequence
\[
0\rightarrow \mathcal{O}_X (K_X+mL)\otimes \mathcal{I}(X,\Delta)\rightarrow
\mathcal{O}_X (K_X+mL)\rightarrow \frac{\mathcal{O}_X (K_X+mL)}{\mathcal{I}(X,\Delta)\otimes \mathcal{O}_X (K_X+mL)}\rightarrow 0
\]
Since $\operatorname{LLC}(X,\Delta,x)=\{x\}$, we have that $\mathcal{O}_x$ is a direct summand of 
\[
\frac{\mathcal{O}_X (K_X+mL)}{\mathcal{I}(X,\Delta)\otimes \mathcal{O}_X (K_X+mL)}
\]
Therefore, by taking the associated long exact sequence and by using Theorem \ref{nadel}, we get a surjection
\[
H^0(X, \mathcal{O}_X (K_X+mL))\rightarrow k 
\]
which is what we wanted. Now suppose that $K_X+mL$ separates all $r-1$ points.
Fix any set $S$ of $r$ points of $X$. Again, let $t$ and $\Delta\in|L|_{\mathbb{Q}}$ be as in Theorem \ref{main} with this choice of $S$ and let $x$ be a point of $S$ such that
$\operatorname{LLC}(X,\Delta,x)=\{x\}$.
Consider once more the short exact sequence above. We have a splitting:
\[
\frac{\mathcal{O}_X (K_X+mL)}{\mathcal{I}(X,\Delta)\otimes \mathcal{O}_X (K_X+mL)} = 
\mathcal{O}_x \oplus \frac{\mathcal{O}_X(K_X+mL)}{\mathcal{I}'\otimes \mathcal{O}_X(K_X+mL)}
\]

By lifting a section that is $0$ on the second factor and $1$ on the first factor, we get a section $s\in H^0(X, \mathcal{O}_X(K_X+mL))$ that vanishes
on $S\setminus \{x\}$ and that does not vanish on $x$. Since $K_X + mL$ separates points of $S\setminus \{x\}$ by induction, we have that $K_X+mL$ separates points of $S$.
Since $S$ is arbitrary, the statement follows.
\end{proof}

\begin{proof}(of Theorem \ref{intro1} and Theorem \ref{intro3}). 
Immediate from Theorem \ref{gg} and Theorem \ref{upperbound}.
\end{proof}

We also record here the following result. 

\begin{corollary}
Let $X$ be a smooth projective threefold and let $L$ be an ample line bundle. Then $|K_X+5L|$ defines a birational morphism.
\end{corollary}
\begin{proof}
This is a consequence of Theorem \ref{gg} and the fact that $F(3,2)<5$ by Table 1. 
\end{proof}

\subsection{Sixfolds}\label{sectionsixfolds}

Table 1 in Section \ref{sectionoptimization} shows that for $n\leqslant 4$ the values of the function $F(n,1)$ are enough to prove Fujita's basepoint freeness conjecture. 
For larger values of $n$, however, the geometry of the problem is not fully reflected in the combinatorics of $F$. In fact, it is possible to carry out a slightly finer study by sharpening the inequalities
appearing in the proof of Theorem \ref{main} in certain geometric situations. This was for instance done in \cite{yezhu1} and \cite{yezhu2} to prove Fujita's freeness conjecture for $n=5$.
This kind of study does not change the asymptotic behavior in $n$, so we only carry it out here for $n=6$ as an example.

Let $s<n$ be two positive integers. Consider the set $U_s \subseteq \mathbb{R}^{s+2}\times \mathbb{N}^{s+2}\times \mathbb{N}^{s+2}$ consisting of elements
\[
(\underline{b}, \underline{d}, \underline{m}) = (b_0, \cdots, b_{s+1}, d_0, \cdots, d_{s+1}, m_0, \cdots, m_{s+1})
\] 
satisfying the following conditions:
\begin{enumerate}
\item $0= b_{s+1}  < b_s < \cdots < b_1 < b_0=n$.
\item $0= d_{s+1} < d_s < \cdots < d_1<d_0=n$.
\item $b_i \leqslant d_i$ for all $1\leqslant i\leqslant s$.
\item $m_i \leqslant \binom{n-\lceil b_i \rceil}{n-d_i}$ for all $1\leqslant i\leqslant s$.
\item $b_i\leqslant 2/m_i$ if $d_i=2$.
\item $d_1\neq n-1$.
\end{enumerate}

Consider now the functions:
\[
g(\underline{b}, \underline{d}, \underline{m},n)=\sum_{i=0} ^s (b_i - b_{i+1}) \sqrt[d_i]{m_i}
\]

and

\[
G(n)=\operatorname{max}\{g(\underline{b}, \underline{d}, \underline{m}, n)| (\underline{b}, \underline{d}, \underline{m})\in \cup_{s=0} ^{n-1} U_s\}
\]

\begin{lemma}\label{finebound}
We have that $G(6)<8$.
\end{lemma}
\begin{proof}
Fix $\underline{d}$ and $\underline{m}$. Then $\overline{U_s} \cap \mathbb{R}^{s+2} \times \{\underline{d}\} \times \{\underline{m}\}$ is a compact convex polyhedron and therefore it's the convex hull of its vertices. By linearity of $g$ in the $b_i$ entries, $g$ is maximal at one such vertex $(\underline{b}, \underline{d}, \underline{m})$. 
By looking at the conditions $(1)-(6)$ defining $U_s$ we see that $\underline{b}$ consists of integral entries, unless $2\in \{d_i\}$. Suppose then that $2\in\{d_i\}$ and let $i$ be the index such that $d_i=2$. 
By possibly erasing all entries (if any) $(b_j, d_j)$ such that $b_j=b_i$ for $j<i$, we may assume that $i$ is the smallest index such that
$b_i$ is a non-integral entry in $\underline{b}$. Furthermore, we must have $m_i\geqslant 3$ and $b_i = 2/m_i$. In this case, it's immediate to see that 
\[
\sum_{j=i} ^{s+1} (b_j-b_{j+1}) \binom{n-\lceil b_j \rceil}{n-d_j} ^{1/d_j} \leqslant \frac{2}{m_i ^{1/2}}
\]

Therefore $g(\underline{b}, \underline{d}, \underline{m}, n)$ is bounded above by the expression:

\[
\sum_{j=0} ^{i-2} (b_j - b_{j+1}) \binom{n-\lceil b_j \rceil}{n-d_j}^{1/d_j} + (b_{i-1}-2/m_i) \binom{n-\lceil b_{i-1} \rceil}{n-d_{i-1}} ^{1/d_{i-1}} + \frac{2}{m_i ^{1/2}}
\]

In either case, we have once again reduced the problem to a finite number of computations
and the result follows by running a computer program on all the possible combinations. 
\end{proof}

Theorem \ref{main} may be slightly sharpened by the following.

\begin{theorem}
Let $X$ be a smooth projective variety of dimension $n$ and let $D$ be an ample $\mathbb{Q}$-divisor. Let $x$ be a point of $X$.
Suppose that 
\[
D^d \cdot Z\geqslant 1
\]

for all irreducible $d$-dimensional subvarieties $Z$ containing $x$.
Fix any positive rational number $0<\epsilon\ll 1$. 
Then there exists a $\mathbb{Q}$-divisor $\Delta\in |tD|_{\mathbb{Q}}$ such that:
\begin{enumerate}
\item $t < G(n)+\epsilon$.
\item $(X,\Delta)$ is log canonical but not Kawamata log terminal at $x$.
\item $\operatorname{LLC}(X,\Delta,x)=\{x\}$.
\end{enumerate}

In particular, if $D$ is Cartier and $n=6$ then $K_X + mD$ is basepoint free for all $m\geqslant 8$.
\end{theorem}
\begin{proof}
Consider the sequence $(D_i, Z_i, T_i, t_i, d_i, b_i, m_i)$ obtained by applying the proof of Theorem \ref{main}. 
Suppose first that $d_1\neq n-1$. Then $(\underline{b}, \underline{d}, \underline{m})$ belongs to $U_s$. 
In fact, conditions $(1)-(5)$ are clear and $(6)$ follows for example from \cite[Theorem B.1]{yezhu2}. 
Therefore, in this case we are done by relation $(*)$ of the proof of Theorem \ref{main} and by Theorem \ref{gg}.

On the other hand, if $d_1=n-1$, then by the proof of \cite[Theorem 4.4]{helmke2} applied to $(n+\epsilon)D$, we have that
\[
t_1 \leqslant n-b_1 \cdot \sqrt[n-1]{m_1}
\]
If one uses this estimate for $t_1$ and relation $(*)$ for $t_i$ with $i\geqslant 2$, we see that $b_1$ simplifies from the expression. 
Therefore, we may always assume that $(\underline{b}, \underline{d}, \underline{m})$ belongs to $U_s$.

Finally, if $n=6$, then the result follows from Lemma \ref{finebound}.
\end{proof}

\input{appendix}

\bibliography{biblio}
\bibliographystyle{alpha}

\end{document}

%% file: appendix.tex
%\documentclass[12pt,twoside]{amsart}
%\usepackage{amsmath,amsfonts,amssymb,amsthm,enumitem,comment,dirtytalk,tikz-cd}
%
% \newcommand{\blue}[1]{\textcolor{blue}{#1}}
% 
% 
%\newtheorem{theorem}{Theorem}[section]
%
%\newtheorem{proposition}[theorem]{Proposition}
%
%\newtheorem{corollary}[theorem]{Corollary}
%
%\newtheorem{lemma}[theorem]{Lemma}
%
%\newtheorem{notation}[theorem]{Notation}
%
%\newtheorem{remark}[theorem]{Remark}
%
%\newtheorem{example}[theorem]{Example}
%
%\newtheorem{conjecture}[theorem]{Conjecture}
%
%\theoremstyle{definition}
%\newtheorem{definition}[theorem]{Definition}
%
%\pagestyle{plain}
%
%\title{Logarithmic bounds on Fujita's conjecture}
%\author{Luca Ghidelli} \author{Justin Lacini}
%\date{}
%
%\begin{document}
	
\section{Appendix}\label{appendix}

In this appendix, we complement Section \ref{sectionoptimization} by showing that the asymptotic bound of Theorem \ref{upperbound} (1) is optimal. We also show some alternative ways of bounding $F(n,r)$. 

\subsection{Optimality}

Theorem \ref{upperbound} (1) asserts that $F(n,1)=O(n\operatorname{log}\operatorname{log}(n))$. Here we show:

\begin{theorem}\label{lowerbound}
For $n$ large enough
\[
F(n,1)\geqslant \frac{1}{4e} n\operatorname{log}\operatorname{log}(n)
\]
\end{theorem}

%Before doing that, we recall the main points of our proof of  the upper bound. 
%The key estimate is given in Lemma \ref{Wbound} (see also Remark \ref{remarkWbound}), which asserts  that 
%$$
%\binom{n-b}{n-d}^{1/d} \leq \frac n { b W(n/b)}
%$$
%where  $W$ denotes Lambert's productlog function (see Definition \ref{definitionW}). Then the double logarithmic bound of Theorem \ref{upperbound} essentially follows from the observation that
%$$
%\sum_{b=1}^n \frac n {b W(n/b)} \sim n \log \log n.
%$$

%In order to show that our estimates are essentially optimal, it suffices to have estimates of the same quality in the opposite direction. 
%The proofs suggest that $f(\underline b, \underline d, n,1)$ is maximal when the sequences $\underline b$ and $\underline d$ satisfy the relation  
%$$
%d_j \approx b_j + b_j W(n/b_j).
%$$
%The proofs also show that most of the contribution to the value of $f$ come from the $b_j$ that are small compared to $n$. 
%To make these observations more precise, we may use the following estimate.
 
%(the number 100 throughout the statement is simply a convenient choice of a large enough absolute constant)

We start with the following. 

\begin{lemma}\label{estimatelowerbound}
	Let $b\leq d\leq n$ be positive integers such that $ b \leq n/10$ and 
	%\quad  \text{and} \quad
	$$
	bW(n/b) \leq  d-b\leq  2bW(n/b).$$
	Then 
	$$
	\binom{n-b}{n-d}^{1/d} \geq \frac 1 {4e}\frac n { b W(n/b)}
	$$ 
\end{lemma}

\begin{proof}
	%W(10) = 1.7455
	First, note that the function $b\mapsto bW(n/b)$ is strictly increasing. Using the fact that $W(10)<2$, we see that
	$$ d\leq b+2bW\left(\frac n b \right) \leq \frac n {10} (1+2W(10))<\frac n 2.$$
	Therefore, we may estimate $\binom {n-b}{n-d}$ as
	$$
	\binom {n-b}{n-d} \geq  \left(\frac {(n/2)}{d-b}\right) ^{d-b}	
	$$
	Let $w=W(n/b)$ and note that
	$$
	\left( \frac {en }{bw}\right)^{w/(w+1)}  = \frac {n}{bw}
	$$
	because both sides are equal to $e^w$. 
	Then:
	\begin{align*}
		\binom{n-b}{n-d}^{1/d} 
		&\geq 
		\left(\frac{n}{2(d-b)}\right) ^{(d-b)/d}\\
		& \geq
		\left(\frac 1 {4e}\frac {en}{bw}\right) ^{w/(w+1)}\\
		& =
		\left(\frac 1 {4e}\right) ^{w/(w+1)} \frac {n}{bw},
	\end{align*}
	which is at least $(4e)^{-1}  n/(bw)$. 
\end{proof}

Lemma \ref{estimatelowerbound} allows us to state sufficient conditions on $(\underline b, \underline d)$ for which $f(\underline b, \underline d, n,1)$ grows asymptotically as $n\log \log n$. 

\begin{lemma}\label{lemmalowerbound}
	Let $n\geq 10$ and let
	$$(b_{s+1},\ldots, b_1) = (0, 1, \ldots ,\lfloor n/10\rfloor )$$ 
	be the whole reverse sequence of integers up to $n/10$. 
	Now suppose that $(\underline b,\underline d)\in R_{s,n}$ satisfies the following condition:
	$$
	 b_j W(n/b_j) \leq  d_j- b_j \leq 2 b_j W(n/b_j)
	$$
	for all $1\leq j\leq s$. Then 
	$$
	f(\underline b, \underline d, n,1) \geqslant \frac 1 {4e} n \log \log n
	$$
\end{lemma}
\begin{proof}
	By Lemma \ref{estimatelowerbound} we have 
	\begin{align*}
		f(\underline b, \underline d , n, 1) 
		& = n- \lfloor n/10\rfloor + \sum_{j=1}^{s} \binom{n-b_j}{n-d_j}^{1/d_j}\\
		& \geqslant \frac 9 {10} n + \frac 1 {4e}\sum_{b=1}^{\lfloor n/10\rfloor} \frac n { b W(n/b)}\\
		& \geqslant \frac 9 {10} n + \frac 1 {4e}\int_{1}^{n/10} \frac n {  b W(n/b)} db
		\end{align*}	
	Then by the substitution $z=n/b$ we get 	
		\begin{align*}	
			 \int_{1}^{n/10} \frac n { b W(n/b)} db
		& = n \int_{10}^{n}\frac  z {  W(z)} dz\\
		& = n \Big[ \log W(z) - \frac 1 {W(z)} \Big]_{10}^{n}\\
	%	& \geqslant n \log W(n/100) - n\log W(100)
	\end{align*}
	Note that  $-1/W(n)+1/W(10)\geqslant 0$ and $\log W(10)<\log 2$.
	Therefore
	$$f(\underline b, \underline d, n,1)\geq \frac n {4e} \log W(n) + \frac 9 {10} n- \frac {\log 2}{4e} n$$
	Since the inequality  
	$$
	%W(z)\geq \log z - \log \log z  %for z>e
	W(n)\geq \frac 1 2 \log n 
	$$
	holds for all $n$, and since $\log 2/(4e)<0.4$, we are done. 
\end{proof}

We are now ready to prove Theorem \ref{lowerbound}.

\begin{proof}(of Theorem \ref{lowerbound}).
	It suffices to show that for $n$ large enough there exist $(\underline b,\underline d)\in R_{s,n}$ as in Lemma \ref{lemmalowerbound}.
	To this aim, let $n\geq 110$, let $(b_{s+1},\ldots, b_1) = (0, 1, \ldots ,\lfloor n/10\rfloor )$ and let 
	$$d_j = b_j + \lceil b_j W(n/b_j)\rceil.$$
	Note that $d_j \leq  b_j + 2 b_j W(n/b_j)$ for all $1\leq j\leq s$ because
	$$   b_j W(n/b_j) \geqslant W(10)>1.$$
	It now suffices to show that
	$$
	0= d_{s+1} < d_s < \cdots < d_1<d_0=n.
	$$
	The inequality $d_1<n$ follows from the the following computation:
	\begin{align*}
		d_1 
		&\leq  \frac n {10}+ \frac n {10} W(11) + 1 \\
		& < \frac {3n} {10}+1 \leq  n,
	\end{align*}
	where we tacitly used inside $W$ the estimate $\lfloor n/10\rfloor \geqslant n/11$, valid for all $n\geq 110$. 
	Finally, note that the function 
	$$
	\delta(b) = b W(n/b)
	$$
	has derivative given by the formula
	$$
	\delta'(b) = \frac {W(n/b)^2 } {W(n/b)+1}.
	$$
	If $b\leq n/10$ then $W(n/b)\geq W(10)>1.7$ and so
	$\delta '(b)>1$. 
	This implies that for all $1\leq j\leq s-1$ we have
	$$d_j-d_{j+1} \geq 2$$
	and so in particular $d_{j+1}<d_j$.
\end{proof}

\begin{remark}
	It is well-known that it is possible to get better estimates on Fujita's conjecture if in the inductive process the dimension of the log canonical centers decreases only by one (i.e. if $d_{j+1} = d_j-1$ for some $j$). 
	See for example \cite{fujita2,kawamata1,helmke2,yezhu1,yezhu2} and see also condition (5) in the definition of $U_{s}$ in Section \ref{sectionsixfolds}. 
	In the proof of Theorem \ref{lowerbound}, however, we have constructed an element $(\underline b,\underline d)$ for which $f(\underline b, \underline d,n,1)$ grows as $n\log \log n$ and such that
	$$ d_{j+1}\leq d_j - 2$$
	for all $1\leq j\leq s$. In particular this shows that a study like the one carried out in Section \ref{sectionsixfolds} does not change the asymptotic behavior of $F$.
\end{remark}

\subsection{Other estimates}

%In this paper we have not attempted to compute the precise best-possible estimates on $F(n,r)$ that could be derived with our method, and focussed more  on the qualitative growth of these estimates. However, it might be useful to record a few 

We record the following elementary refinement of Lemma  \ref{elementaryestimate} which is enough to give logarithmic bounds on the problem of separation of $r$ points.

\begin{lemma}\label{newelementaryestimate}
	Let $b\leq d\leq n $ and $r$ be positive integers. Then 
	$$\left[ r\binom{n-b}{n-d}\right] ^{1/d} \leq \sqrt[b]r  + e\frac n b - e $$
\end{lemma}

\begin{proof}
	By the basic version of Stirling's inequality $A!\geqslant (A/e)^A$, we have that 
	$$\binom{n-b}{n-d}^{1/(d-b)}\leq \frac {e(n-b)}{d-b}.$$ 
	Then, using Young's inequality  $A^\lambda B^{1-\lambda} \leq \lambda A + (1-\lambda) B$ we get
	\begin{align*}
		r^{1/d} \binom{n-b}{n-d}^{1/d} 
		& = (\sqrt[b] r)^{b/d} \cdot \left(\binom{n-b}{n-d}^{1/(d-b)}\right)^{1- b /d}\\
		& \leq  \frac b d \sqrt[b] r + \left(1 - \frac b d\right) \binom{n-b}{n-d}^{1/(d-b)}\\
		& \leq  \frac b d \sqrt[b] r + \frac e d (n-b)\\
		& \leq  \frac b b \sqrt[b] r + \frac e b (n-b)\\
		& = \sqrt[b]r  + e\frac n b - e.
	\end{align*}
\end{proof}
\begin{corollary}\label{newelementarycorollary}
	Let $n,r$ be positive integers. Then
	$$
	F(n,r)\leq  e n \log n + \sum _{b=1}^{n} \sqrt[b]r 
	$$
\end{corollary}
\begin{proof}
	Let $(\underline b, \underline d)\in R_{s,n}$. By Lemma \ref{sumbound} and Lemma \ref{newelementaryestimate} we have
	\begin{align*}
		f(\underline b, \underline d, n,r) 
		&\leq   en\left(-1+ \sum _{b=1}^{n} \frac 1 b\right) +\sum _{b=1}^{n} \sqrt[b]r\\
		& \leq  e n \log n + \sum _{b=1}^{n} \sqrt[b]r.
	\end{align*}
\end{proof}

We conclude by including the following simple estimates from below.

\begin{proposition}\label{easylowerboundsinr}
	Let $n,r$ be positive integers. Then
	$$ 
	F(n,r)\geq \max \left\{ \frac {r^{1/n}}{4e} n \log \log n, \sum _{b=1}^{n} \sqrt[b]r \right\}
	$$
\end{proposition}

\begin{proof}
	The first lower bound follows from Corollary \ref{lowerbound} and the simple observation that 
	$$ F(n,r)\geq r^{1/n} F(n,1).$$
	To prove the second lower bound, let 
	$$\underline b = \underline d  = (n,n-1,\ldots,1,0)$$
 	Then 
 	$$
 	F(n,r)\geq f(\underline b, \underline d, n,r) = \sum _{b=1}^{n} \sqrt[b]r.
 	$$
\end{proof}

If one fixes $n$ and lets $r$ be large enough, it is possible to compute $F(n,r)$ exactly. 
 
\begin{proposition}
	Fix a positive integer $n$. Then for each large enough positive integer $r$ we have
	$$ F(n,r) = \sum _{b=1}^{n} \sqrt[b]r. $$
\end{proposition}
\begin{proof}
	By Proposition \ref{easylowerboundsinr} it is sufficient to prove the inequality $F(n,r) \leq \sum _{b=1}^{n} \sqrt[b]r$. Let  $(\underline b, \underline d)\in R_{s,n}$. By Lemma \ref{sumbound} we have 
	\[
	f(\underline{b},\underline{d},n,r) 
	\leqslant 
	\sum_{b=1} ^n \left[ r\binom{n-b}{n-d(b)}\right] ^{1/d(b)}
	\]
	for some $b\leq d(b)\leq n$. 
	If $r$ is large enough then each term of the sum is dominated, respectively, by a $b$-th root of $r$:
$$	 r^{1/d(b)}\left[ \binom{n-b}{n-d(b)}\right] ^{1/d(b)} 
\leq 
r^{1/b} $$
\end{proof}

%% file: log.bbl
\begin{thebibliography}{Kaw97}

\bibitem[AS95]{angehrnsiu}
Urban Angehrn and Yum~Tong Siu.
\newblock Effective freeness and point separation for adjoint bundles.
\newblock {\em Invent. Math.}, 122(2):291--308, 1995.

\bibitem[Ein97]{ein}
Lawrence Ein.
\newblock Multiplier ideals, vanishing theorems and applications.
\newblock In {\em Algebraic geometry---{S}anta {C}ruz 1995}, volume~62 of {\em
  Proc. Sympos. Pure Math.}, pages 203--219. Amer. Math. Soc., Providence, RI,
  1997.

\bibitem[EL93]{einlazarsfeld1}
Lawrence Ein and Robert Lazarsfeld.
\newblock Global generation of pluricanonical and adjoint linear series on
  smooth projective threefolds.
\newblock {\em J. Amer. Math. Soc.}, 6(4):875--903, 1993.

\bibitem[Fuj87]{fujita1}
Takao Fujita.
\newblock On polarized manifolds whose adjoint bundles are not semipositive.
\newblock In {\em Algebraic geometry, {S}endai, 1985}, volume~10 of {\em Adv.
  Stud. Pure Math.}, pages 167--178. North-Holland, Amsterdam, 1987.

\bibitem[Fuj93]{fujita2}
Takao Fujita.
\newblock Remarks on ein-lazarsfeld criterion of spannedness of adjoint bundles
  of polarized threefolds.
\newblock {\em arXiv preprint math/9311013}, 1993.

\bibitem[Hei02]{heier}
Gordon Heier.
\newblock Effective freeness of adjoint line bundles.
\newblock {\em Doc. Math.}, 7:31--42, 2002.

\bibitem[Hel97]{helmke1}
Stefan Helmke.
\newblock On {F}ujita's conjecture.
\newblock {\em Duke Math. J.}, 88(2):201--216, 1997.

\bibitem[Hel99]{helmke2}
Stefan Helmke.
\newblock On global generation of adjoint linear systems.
\newblock {\em Math. Ann.}, 313(4):635--652, 1999.

\bibitem[Kaw97]{kawamata1}
Yujiro Kawamata.
\newblock On {F}ujita's freeness conjecture for {$3$}-folds and {$4$}-folds.
\newblock {\em Math. Ann.}, 308(3):491--505, 1997.

\bibitem[Kol97]{pairs}
J\'{a}nos Koll\'{a}r.
\newblock Singularities of pairs.
\newblock In {\em Algebraic geometry---{S}anta {C}ruz 1995}, volume~62 of {\em
  Proc. Sympos. Pure Math.}, pages 221--287. Amer. Math. Soc., Providence, RI,
  1997.

\bibitem[Laz04]{lazarsfeld2}
Robert Lazarsfeld.
\newblock {\em Positivity in algebraic geometry. {II}}, volume~49 of {\em
  Ergebnisse der Mathematik und ihrer Grenzgebiete. 3. Folge. A Series of
  Modern Surveys in Mathematics [Results in Mathematics and Related Areas. 3rd
  Series. A Series of Modern Surveys in Mathematics]}.
\newblock Springer-Verlag, Berlin, 2004.
\newblock Positivity for vector bundles, and multiplier ideals.

\bibitem[Rei88]{reider}
Igor Reider.
\newblock Vector bundles of rank {$2$} and linear systems on algebraic
  surfaces.
\newblock {\em Ann. of Math. (2)}, 127(2):309--316, 1988.

\bibitem[YZ17]{yezhu1}
Fei Ye and Zhixian Zhu.
\newblock Global generation of adjoint line bundles on projective 5-folds.
\newblock {\em Manuscripta Math.}, 153(3-4):545--562, 2017.

\bibitem[YZ20]{yezhu2}
Fei Ye and Zhixian Zhu.
\newblock On {F}ujita's freeness conjecture in dimension 5.
\newblock {\em Adv. Math.}, 371:107210, 56, 2020.
\newblock With an appendix by Jun Lu.

\end{thebibliography}
